\documentclass[12pt]{article}
\usepackage{latexsym,amssymb,amsmath,ulem}
\usepackage{color}
\usepackage{graphicx}
\usepackage{latexsym,amssymb,amsmath,ulem}
\usepackage{amsfonts}
\usepackage{latexsym}
\usepackage{graphicx}
\usepackage{color}
\usepackage[colorlinks=true,linkcolor=blue,citecolor=green]{hyperref}	

\def\K{\mathcal{K}}

\def\liminf{\mathop{{\rm lim}\,{\rm inf}}}

\def\tto{\;{\lower 1pt \hbox{$\rightarrow$}}\kern -10pt
\hbox{\raise 2pt \hbox{$\rightarrow$}}\;}
\def\e{\epsilon}

\def\ra{\rangle}
\def\la{\langle}

\def\R{I\!\!R}
\def\N{I\!\!N}
\def\ox{\bar{x}}

\def\oz{\bar{z}}

\def\ovr{\bar{r}}

\def\epi{\mbox{\rm epi}\,}
\def\hypo{\mbox{\rm hypo}\,}

\def\dom{\mbox{\rm dom}\,}

\def\spa{\mbox{\rm span}\,}

\def\cor{\mbox{\rm core}\,}
\def\acl{\mbox{\rm acl}\,}
\def\lina{\mbox{\rm lina}\,}

\def \N{I\!\!N}

\newtheorem{theorem}{Theorem}[section]
\newtheorem{proposition}{Proposition}[section]

\newtheorem{lemma}{Lemma}[section]
\newtheorem{remark}{Remark}[section]
\newtheorem{example}{Example}[section]

\numberwithin{equation}{section}
\renewcommand{\theequation}{\thesection.\arabic{equation}}
\normalsize
\input{tcilatex}

\def\ra{\rangle}
\def\la{\langle}

\def\la{\langle}
\def\ra{\rangle}

\begin{document}

\title{Algebraic Farkas Lemma and Strong Duality for Perturbed Conic Linear Programming 
 }
\author{Pham Duy Khanh\thanks{%
 Group of Analysis and Applied Mathematics, Department of Mathematics, Ho Chi Minh City
 University of Education, Ho Chi Minh City, Vietnam (khanhpd@hcmue.edu.vn). }\quad Vu Vinh Huy Khoa\thanks{%
 Department of Mathematics, Wayne State University, Detroit, Michigan, USA (khoavu@wayne.edu).}   \quad Tran Hong Mo\thanks{%
Ho Chi Minh City Open University,  Vietnam
(mo.th@ou.edu.vn).} }
\maketitle
\date{}

\noindent
{\small{\bf Abstract} This paper addresses the study of algebraic versions of Farkas lemma and strong duality results in the very broad setting of infinite-dimensional conic linear programming in dual pairs of vector spaces. To this end, purely algebraic properties of perturbed optimal value functions of both primal and dual problems and their corresponding hypergraph/epigraph are investigated. The newly developed hypergraphical/epigraphical sets, inspired by Kretschmer's  closedness conditions \cite{Kretschmer61}, together with their novel convex separation-type characterizations, give rise to various perturbed Farkas-type lemmas which allow us to derive complete characterizations of ``zero duality gap''. Principally, when certain structures of algebraic or topological duals are imposed, illuminating implications of the developed condition are also explored. \\[1ex]
}

\section{Introduction}
\textit{Infinite-dimensional conic linear programming} arises in the literature of optimization as a natural extension of linear programming, carrying the existing setting to infinite-dimensional vector spaces while requiring a feasible point to lie within a convex cone - a generalization of the well-known nonnegative orthant in a Euclidean space emanating from linear programming. A variety of intriguing optimization problems can be expressed within this framework, including potential theory \cite{Yamasaki68}, continuous-time linear programming and continuous-time network flow \cite{AndersonNash87}, and approximation theory \cite{Krabs}. Unlike the classical linear programming theory, in the broader context of \textit{infinite-dimensional conic linear programming}, no comprehensive theory has been established, and efforts in duality theory, for example, have not yet been unified to the same degree.

The first attempt to study linear programming in more general settings of linear spaces dates back to Duffin \cite{Duffin56}, in which the primal space $X$ is a locally convex topological space while the dual space coincides with the topological dual of $X$. Kretschmer \cite{Kretschmer61} extends the exploration with his nominal work on duality theory and existence theorems within the broader framework of \textit{dual pair}, which aligns with the formulation presented in Yamasaki \cite{Yamasaki68,Yamasaki79}, and specifically, of our study. 

\textit{Farkas lemma} is a typical example of the very nice collection of \textit{theorems of the alternative} which, generally speaking, characterizes the consistency (or solvability) of systems of linear inequalities and their corresponding duals. The primary significance of \textit{Farkas lemma} in the theory of linear programming lies in its role in establishing the \textit{strong duality theorem}. A comprehensive review of Farkas lemma, duality theory and related topics can be found, e.g., from \cite{Anderson83,AndersonNash87,DinhJeyakumar} and the references therein. In greater detail, Shapiro \cite{Shapiro} employed the theory of convex conjugate duality and convex subdifferential pioneered by Rockafellar to present a novel approach to exploring ``zero duality gap'' and existence results via analyzing the conjugate of the objective functions. The reader is referred to Z\u alinescu \cite{Zalinescu08,Zalinescu23} for manifestation of techniques and results of the same manner, and further, several alternative proofs for some well-known existing results in the subject. Vinh et al. \cite{VinhKimTamYen16} proposes a thorough study on the so-called duality gap function, presenting various ``zero duality gap'' results based on their interior conditions. Further, Luan and Yen \cite{LuanYen23} addresses solution existence for both primal and dual problem, obtaining a Farkas-type result based on a generalized Slater condition in the setting of locally convex topological vector spaces. Quite recently, Khanh et al. \cite{KhanhMoTrinh} establishes necessary and sufficient conditions (weakest conditions) ensuring both the ``zero duality gap'' and solvability of the problems via their closedness-like conditions, stemming from the very classical \textit{closedness condition}; see Anderson and Nash \cite[Theorem~3.10]{AndersonNash87}. Many sufficient conditions guaranteeing the closedness condition can also be found from \cite{AndersonNash87}, e.g., interior point conditions, boundedness/compactness conditions.

Separating from the previously mentioned findings, the present article provides an in-depth analysis of numerous perturbed variants of Farkas lemma and strong duality theorem merely in the very general setting of dual pair of vector spaces without topological structures. We start with a detailed investigation on cost-perturbed versions of the primal and dual problem, whose optimal value functions are equipped with nicely behaved hypergraphical and epigraphical set, respectively. The constructions of these (hyper/epi)graphical sets are motivated by both the Kretschmer's closedness conditions and their links to strong duality theorems (see \cite{Kretschmer61} and \cite{KhanhMoTrinh}). Bonding connections as well as distinguished features between our innovative graphical sets and the well-known auxiliary sets from the theory of strong duality \cite{KhanhMoTrinh} are investigated and illustrated through the celebrated example of Gale. We also succeed in deriving characterizations of the epigraphical/hypergraphical sets in the language of convex separations in convex analysis (i.e., extension theorems in functional analysis). The ingredients collected grant us cutting-edge \textit{verifiable} necessary and sufficient conditions, which are later named conditions {\bf (D)} and {\bf (D$^*$)}, for the characteristic alternative arguments in the Farkas lemma, as well as the ``zero duality gap'', thus offering a unified perspective on the theory of the Farkas lemma and strong duality results. In the particular case of algebraic dual pairs $(X,X^*)$, we employ the notion of \textit{algebraic cores} of sets and basic algebraic separation theorems to provide sufficient conditions ensuring the validation of {\bf (D)} and {\bf (D$^*$)}. Last but not least, the equivalence between the latter and a closedness-type condition similar to that from \cite{KhanhMoTrinh} is assured when topological structures are finally equipped to the spaces in consideration.

The rest of the paper is organized as follows. Preliminaries on dual pairs of vector spaces and the standing formulations of our main problems are given in Section \ref{sec:prelim}. Section \ref{sec:perturbed} studies perturbed conic linear programming problems, motivates these perturbations, develops the associated optimal value functions, and introduces our verifiable necessary and sufficient condition via the former ingredients. Section \ref{sec:strongdual} is devoted to the study of Farkas lemma and strong duality theory under the newly proposed algebraic criterion. Section \ref{sec:algebra} and Section \ref{sec:topo} showcase respectively, an algebraic sufficient condition guaranteeing the mentioned criterion and its full characterization in the setting of topological vector spaces. The concluding Section~\ref{sec:con} summarizes the major contributions of the paper.

\section{Preliminaries}\label{sec:prelim}
The following facts involving dual pairs of vector spaces can be found from Schaefer \cite{Schaefer71}.

Given that $X$ and $Y$ are real vector spaces and $\la \cdot,\cdot\ra$ is a \textit{bilinear form} on $X\times Y$, i.e., $\la x,\cdot\ra$ is linear on $Y$ for any $x\in X$ and $\la \cdot,y\ra$ is linear on $X$ for any $y\in Y$. Then $(X,Y)$ is called a \textit{dual pair}. If in addition $\la \cdot,\cdot\ra$ satisfies the \textit{separation axioms}:

(i) $\la x_0,y\ra =0$ for all $y\in Y$ implies $x_0=0_X$,

(ii) $\la x,y_0\ra =0$ for all $x\in X$ implies $y_0=0_Y$,\\
then $(X,Y)$ is called a \textit{separated dual pair}. As an example, knowing that each vector space possesses a Hamel basis containing a given nonzero vector, one is able to show that $X$ together with its algebraic dual $X^*$ forms a separated dual pair. 

Let $(X, Y)$ and $(Z, W)$ be two separated dual pairs of vector spaces. For simplicity, the bilinear forms on $X$ and $Y$, and on $Z$ and $W$ will both be denoted by $\langle \cdot,\cdot\rangle$. Let $A$ be a linear map from $X$ to $Z$. The \textit{adjoint (transpose)} of $A$ is the linear map
$A^*:W\rightarrow Y$ defined via the condition
\begin{equation}\label{adjoint-pure}
\langle x, A^*w\rangle=\langle Ax,w\rangle\quad \forall x\in X,\;\forall w\in W.
\end{equation}
Let $P\subset X$ be a convex cone, i.e., $P$ is nonempty and $\alpha x + \beta y \in P$ for any $\alpha,\beta \ge 0$ and $x,y\in P$. The \textit{(positive) dual cone} of $P$ is defined as
$$
P^*:=\left\{y\in Y:\langle x,y\rangle\geq 0\quad \forall x\in P\right\}.
$$
Since $\la Y,X\ra$ is also a dual pair given that $\la X,Y\ra$ is, we may define the  \textit{(positive) bidual cone} of $P$ as follows:
\begin{equation}\label{P**}
    P^{**}:=\{x\in X:\la x,y\ra \ge 0 \quad \forall y\in P^*\}.
\end{equation}
We are interested in the following \textit{algebraic inequality constrained conic linear programming problem} and its dual (see \cite{Kretschmer61,Shapiro,Yamasaki68,Yamasaki79,Zalinescu78}): 
\begin{eqnarray}
    (\mathcal{P})&&\qquad \min\{\langle x,c\rangle:Ax-b\in Q, x\in P\}, \label{P}\\
    (\mathcal{D})&&\qquad \max\{\langle b,w\rangle:-A^*w+c\in P^*, w\in Q^*\}, \label{D}
\end{eqnarray}
where $P\subset X$ and $Q\subset Z$ are convex cones, and {$b$ and $c$ are given elements of $Z$ and $Y$, respectively.} 
There is also an \textit{equality constrained} form (see \cite{Anderson83}, \cite[Sect. 3.3]{AndersonNash87}, and \cite{VinhKimTamYen16}) of the problem ($\mathcal{P}$) as follows:
\begin{equation}(\mathcal{EP})\qquad\min\{\langle x,c\rangle:Ax-b=0_Z, x\in P\}, \label{EP}
\end{equation}
and its corresponding dual is  
\begin{equation}(\mathcal{ED})\qquad\max\{\langle b,w\rangle:-A^*w+c\in P^*, w\in W\}.\label{ED}
\end{equation}
{One can readily verify that either $(\mathcal{EP})$ or $(\mathcal{P})$ can be converted into the other. Indeed, problem $(\mathcal{P})$
becomes a program in the form $(\mathcal{EP})$ if $Q$ is the cone consisting only of the origin. Conversely, $(\mathcal{EP})$ can be written in the form $(\mathcal{P})$ by replacing $X$ with $X\times Z$, $P$
with $P \times Q$, and $A$ with the linear map sending $(x, z)$ to $Ax - z$. Thus there is no
loss of generality in dealing entirely with either $(\mathcal{EP})$ or $(\mathcal{P})$.
In what follows, we present results for both formulations.}

Let $F(\mathcal{P}):= \{x\in P: Ax - b \in Q\}$ and $F(\mathcal{D}) := \{w\in Q^*: -A^*w+c\in P^*\}$ denote the sets of feasible solutions of $(\mathcal{P})$  and $(\mathcal{D})$, respectively. Our standing assumption is that ($\mathcal{P}$) and ($\mathcal{D}$) are \textit{consistent}, i.e., $F(\mathcal{P})$ and $F(\mathcal{D})$ are nonempty. We use the notations 
\begin{equation*}
\begin{aligned}
\textup{val}(\mathcal{P})&:=\inf\{\la x,c\ra : Ax-b\in Q,x\in P\} \text{ and}\\
\textup{val}(\mathcal{D})&:=\sup\{\la b,w\ra: -A^*w+c\in P^*,w\in Q^*\}
\end{aligned}
\end{equation*}
for the \textit{optimal values} of ($\mathcal{P}$) and ($\mathcal{D}$), respectively. A result that assures the identity
\begin{equation}\label{strongdual}
   \text{val($\mathcal{P}$)}=\text{val($\mathcal{D}$)}
\end{equation}
is called a \textit{strong duality theorem}. The following \textit{weak duality} result can be found in Kretschmer \cite[Theorem~1]{Kretschmer61}.
\begin{theorem}\label{theo:weak-dual}{\rm {\bf (Weak duality)}} We always have \textup{val($\mathcal{P}$)}$\;\ge\;$\textup{val($\mathcal{D}$)}. If \textup{($\mathcal{P}$)} and \textup{($\mathcal{D}$)} are both consistent, then both values are finite. 
    
\end{theorem}

As illustrated via Gale's example in \cite[Section~4.1]{VinhKimTamYen16}, one of the main differences between conic linear programming and linear programming is that even if both the primal problem \eqref{P} and the dual problem \eqref{D} are \textit{solvable} (i.e., each has at least one optimal solution), it may happen that val($\mathcal{P}$)$\;>\;$val($\mathcal{D}$).

To efficiently study the strong duality theory for ($\mathcal{P}$) and ($\mathcal{D}$), we take into account the following auxiliary sets whose equality-constrained-counterparts appeared for the first time in Khanh et al. \cite{KhanhMoTrinh}. Let the sets $\mathcal{H}\subset Z\times\R$ and $\mathcal{K}\subset Y\times\R$ be as follows:
\begin{equation}\label{Definition_H}
\mathcal{H}:=\displaystyle\bigcup_{x\,\in\, P}\big((Ax-b-Q)\times[\langle x,c\rangle,+\infty[\big),
\end{equation}
\begin{equation}\label{Definition_K}
\mathcal{K}:=\displaystyle\bigcup_{w\,\in\, Q^*}\big((-A^*w+c-P^*)\,\times\,]-\infty,\langle b,w\rangle]\big).
\end{equation}
The set $\mathcal{H}$ (resp. $\mathcal{K}$) carries the details of a feasible solution of the primal problem $(\mathcal{P})$ (resp. dual problem $(\mathcal{D})$), including the associated cost function value. Moreover, it is easy to check that 
\begin{equation}\label{valP2}
    \textup{val}(\mathcal{P}) = \inf \{r\in \R: (0_Z,r)\in \mathcal{H}\}
\end{equation}
and 
\begin{equation}\label{valD2}
     \textup{val}(\mathcal{D}) = \sup \{r\in \R: (0_Y,r)\in \mathcal{K}\}.
\end{equation}
Since $(-b-Q)\times\{0\}\subset \mathcal{H}$ and $$\mathcal{H}= B(P \times Q \times \R_+) - \{(b,0)\},$$ where
\begin{equation}\label{H-Pinter}
  B(x,z,\alpha):= \big(Ax-z,\la x,c\ra +\alpha\big), \; (x,z,\alpha)\in X\times Z\ \times \R,
\end{equation}
it follows readily that {$\mathcal{H}$ is nonempty and convex, as it is the image of a convex set under an affine transformation}. The same conclusion also holds for $\mathcal{K}$. In the literature, one is very likely to encounter (see \cite[Section~3]{Kretschmer61} and \cite{Yamasaki68,Yamasaki79}) the sets $\mathcal{H}'\subset Z\times \R$ and $\mathcal{K}'\subset Y\times \R$ defined as
\begin{equation}\label{H'K'}
\begin{aligned}
    \mathcal{H}'&:= \displaystyle\bigcup_{x\,\in\, P}\big((Ax-Q)\times[\langle x,c\rangle,+\infty[\big),\\
    \mathcal{K}'&:= \displaystyle\bigcup_{w\,\in\, Q^*}\big((-A^*w-P^*)\,\times\,]-\infty,\langle b,w\rangle]\big),
\end{aligned}
\end{equation}
which merely turn out to be translations of the sets $\mathcal{H}$ and $\mathcal{K}$ in consideration, respectively, via the relations
\begin{equation*}
\begin{aligned}
    \mathcal{H}'&=\mathcal{H}+\{(b,0)\},\; \text{ and }\; \mathcal{K}'=\mathcal{K}-\{(c,0)\}.
\end{aligned}
\end{equation*}
One of the most well-known sufficient conditions guaranteeing the solvability of ($\mathcal{P}$) as well as the \textit{strong duality result} \eqref{strongdual} is the \textit{closure condition of} $\mathcal{H}'$ from \eqref{H'K'}, see \cite[Corollary~3.1]{Kretschmer61}. In the literature, the \textit{interior point condition} is a more restrictive yet easier-to-verify requirement. When $Q=\{0_Z\}$, two renowned variants of this condition are outlined by Anderson and Nash \cite[Theorem~3.11~and~Theorem~3.12]{AndersonNash87}. The validations of the two stated theorems are intricately tied to the topological properties exhibited by the spaces in question. Notably, \cite[Theorem~3.12]{AndersonNash87} even necessitates the Banach structures of the spaces $X$ and $Z$. A natural question arises: \textit{Is it possible to construct a merely-algebraic sufficient condition for the strong duality \eqref{strongdual}}? We answer this question in the affirmative via condition {\bf (D)} defined in Section \ref{sec:strongdual} below. Roughly speaking, condition {\bf (D)} is shown to be equivalent to the fulfillment of strong duality \eqref{strongdual} for all possible values of $b$ appearing in the objective value function \eqref{D} of $(\mathcal{D})$, the latter of which is referred to as \textit{perturbed strong duality}. 

The next section is devoted to the study of perturbed optimal value functions on which condition {\bf (D)} and its dual counterpart {\bf (D$^*$)} are built.

\section{Algebraic perturbed conic linear programming}\label{sec:perturbed}
Recall from \cite[Theorem~4.4]{KhanhMoTrinh} that, within the framework of dual pairs of topological vector spaces $(X,Y)$ and $(Z,W)$, the condition 
\begin{equation*}
    \mathcal{H}\cap (\{0_Z\}\times \R) = \overline{\mathcal{H}} \cap (\{0_Z\} \times \R)
\end{equation*}
enforces the desired \textit{strong duality} result between $(\mathcal{P})$ and $(\mathcal{D})$. Here, $\overline{\mathcal{H}}$ stands for the closure of $\mathcal{H}$ with respect to the weak topology $\sigma(Z\times \R,W\times \R)$. Inspired by this, we raise the following question: ``\textit{How can we represent $\,\overline{\mathcal{H}}\,$, say, via a suitable set $\mathcal{N}$ within our present general framework of dual pairs of vector spaces, where no topological structure is assumed?''} Importantly, our choice of $\mathcal{N}$--and the ensuing purely algebraic duality results--must be consistent with $\overline{\mathcal{H}}$ and the existing literature once topologies are imposed, for instance with the results in \cite{KhanhMoTrinh}. Building on \cite[Proposition~4.1~and~Theorem~2.2]{KhanhMoTrinh}, it is natural to require 
\begin{equation*}
    \mathcal{N}\cap (\{0_Z\}\times \R) = \{0_Z\}\times \big[\textup{val}(\mathcal{D}),+\infty\big[.
\end{equation*}
Since this identity only encodes information about $\mathcal{N}$ on the slice $\{0_Z\}\times \R$, we henceforth replace $0_Z$ with an arbitrary $z\in Z$ and obtain 
\begin{equation}\label{eq:motivation-N}
    \mathcal{N}\cap (\{z\}\times \R) = \{z\}\times \Big[\sup\big\{\la b+z,w\ra: w\in F(\mathcal{D})\big\},+\infty\Big[.
\end{equation}
The right-hand side of \eqref{eq:motivation-N} is nonempty if and only if 
\begin{equation}\label{eq:z-indom}
\sup\big\{\la b+z,w\ra: w\in F(\mathcal{D})\big\}<\infty.
\end{equation}

Taking the union of \eqref{eq:motivation-N} for all $z$ satisfying \eqref{eq:z-indom}, we arrive at the definition
\begin{equation}\label{N}
    \mathcal{N} := \epi v_{\mathcal{D}},
\end{equation}
where 
\begin{equation}\label{Dy}
v_{\mathcal{D}}(z) := \sup\{\langle b + z, w\rangle : w\in F(\mathcal{D})\}\quad (z\in Z)
\end{equation}
is the optimal value function associated with the \textit{cost-function-perturbed} variant of $(\mathcal{D})$, in which the feasible set is kept fixed. Dually, let us define 
\begin{equation}\label{M}
    \mathcal{M} := \hypo v_{\mathcal{P}},
\end{equation}
where
\begin{equation}\label{Pz}
v_{\mathcal{P}}(y) := \inf\{\langle x, c-y\rangle : x\in F(\mathcal{P})\}\quad (y\in Y).
\end{equation}
The \textit{perturbed} optimal value functions $v_{\mathcal{D}}$ and $v_{\mathcal{P}}$, along with their associated (epi/hyper) --graphical sets $\mathcal{N}$ and $\mathcal{M}$, and the relationships linking 
$\mathcal{N},\mathcal{M}$ to $\mathcal{H},\mathcal{K}$, are the main objects of interest in this section.

First, the concavity of $v_{\mathcal{P}}$ and convexity of $v_{\mathcal{D}}$ are easily recognized. The effective domains $\limfunc{dom}v_{\mathcal{P}} := \{y \in Y : v_{\mathcal{P}}(y) > - \infty\}$ and $\limfunc{dom}v_{\mathcal{D}} := \{z \in Z :  v_{\mathcal{D}}(z) < +\infty\}$ are nonempty convex sets. {Due to the weak duality stated in Theorem \ref{theo:weak-dual}, we have $$v_{\mathcal{P}}(0_Y)=\text{val}(\mathcal{P})\quad \text{ and }\quad v_{\mathcal{D}}(0_Z)=\text{val}(\mathcal{D}),$$ 
these values are finite since $(\mathcal{P})$ and $(\mathcal{D})$ are consistent}. Further, $v_{\mathcal{P}}$ (resp. $v_{\mathcal{D}}$) is \textit{radially upper semicontinuous} (resp. \textit{radially lower semicontinuous}), i.e., upper semicontinuous (resp. lower semicontinuous) with respect to line segments. To be specific for $v_{\mathcal{D}}$, for every $\oz$ and $z\in Z$ it holds that $$v_{\mathcal{D}}(\oz)\le \liminf_{\lambda\rightarrow 0} v_{\mathcal{D}}\big((1-\lambda)\oz + \lambda z\big),$$ 
where the lower limit/limit inferior of an extended real-valued function is defined in \cite[Definition~1.16]{Penot}. This observation is claimed below.
\begin{proposition}\label{radiallycont}
The perturbed optimal value function $v_{\mathcal{P}}$ is radially upper semicontinuous, while $v_{\mathcal{D}}$ is radially lower semicontinuous.
\end{proposition}
\begin{proof}
{Since $v_{\mathcal{D}}$ (resp. $v_{\mathcal{P}})$ is the supremum (resp. infimum) of radially continuous linear functionals, it follows that $v_{\mathcal{D}}$ (resp. $v_{\mathcal{P}})$ is itself radially lower (resp. upper) semicontinuous.}
\end{proof}

Next, we employ a parametric version of Gale's example (see \cite[Section~3.4.2]{AndersonNash87}) to illustrate the computation of the epigraphical set $\mathcal{N}$ from \eqref{N}. The celebrated Gale's example serves in the literature as an instance whence ``positive duality gap'' may happen even in the setting of finite-dimensional conic linear programming. Recently, \cite{KhanhMoTrinh} utilized it to examine whether a sufficient condition for strong duality, established in the same work, is necessary, while many perturbed versions of Gale's example are extensively studied in \cite{VinhKimTamYen16,Zalinescu23}.

\begin{example}\label{exam:Gale1}
\rm 
For real numbers $\alpha$ and $\beta$, we consider the parametric linear problem
$$
\left(\mathrm{PG}_{\alpha, \beta}\right) \quad \min \left\{x_0: x_0+\sum_{i=1}^{\infty} i x_i=\alpha, \sum_{i=1}^{\infty} x_i=\beta, x_i \geq 0, i=0,1,2, \ldots\right\} .
$$

The problem $\left(\mathrm{PG}_{\alpha, \beta}\right)$ is an example of an infinite-dimensional conic linear programming \eqref{P}. Indeed, we have $X:=\R^{(\mathbb{N})}$ consisting of real sequences with finitely many nonzero terms. It is well-known that the dual of $X$ is $Y:=\R^{\mathbb{N}}$, i.e., the space of all real sequences. $(X,Y)$ is then a separated dual pair with respect to the bilinear form
$$
\langle x, y\rangle:=\sum_{i=0}^{\infty} x_i y_i, \quad \forall x=\left(x_0, x_1, x_2, \ldots\right) \in X,\; \forall y=\left(y_0, y_1, y_2, \ldots\right) \in Y.
$$
On the other hand, consider the separated dual pair $(Z,W):=(\R^2,\R^2)$ with respect to the bilinear form
$$
\langle z, w\rangle:=z_1 w_1+z_2 w_2, \quad \forall z=\left(z_1, z_2\right) \in Z,\; \forall w=\left(w_1, w_2\right) \in W.
$$
In $X$, define the convex cone $P=\R_{+}^{(\mathbb{N})}:=\R^{(\mathbb{N})} \cap \R_{+}^{\mathbb{N}}$, where
$$
\R_{+}^{\mathbb{N}}:=\left\{x=\left(x_0, x_1, x_2, \ldots\right) \in \R^{\mathbb{N}}: x_i \geq 0, i=0,1,2, \ldots\right\}.
$$
Clearly, $P^* = \R_{+}^{\mathbb{N}}$. Let $Q:=\{0_Z\}$, $b:=(\alpha, \beta) \in Z$, and $c:=(1,0,0, \ldots) \in Y$. Let $A: X \rightarrow Z$ and its adjoint $A^*: W \rightarrow Y$ be the linear maps represented by the infinite matrices
$$
\begin{gathered}
A=\left(\begin{array}{lllll}
1 & 1 & 2 & 3 & \ldots \\
0 & 1 & 1 & 1 & \ldots
\end{array}\right), \\
A^*=A^T=\left(\begin{array}{lllll}
1 & 1 & 2 & 3 & \ldots \\
0 & 1 & 1 & 1 & \ldots
\end{array}\right)^T.
\end{gathered}
$$
{The dual problem to $\left(\mathrm{PG}_{\alpha, \beta}\right)$ admits the set of feasible solutions as
\begin{equation*}
\begin{aligned}
    F^*&=\big\{w=(w_1,w_2)\in \R^2 : -A^* w + c\in P^*,\, w\in Q^*\big\}\\
    &=\big\{(w_1,w_2)\in \R^2 : -w_1 + 1 \ge 0, -iw_1 - w_2 \ge 0,\, i=1,2,\ldots\big\}. 
\end{aligned}
\end{equation*}}

Now we explore the epigraphical set $\mathcal{N}=\bigcup_{z\,\in\, \mathrm{dom}\, v_{\mathcal{D}}}\big(\{z\}\times [v_{\mathcal{D}}(z),+\infty[\big)$, where {\begin{equation}\label{vDz}
\begin{aligned}
    v_{\mathcal{D}}(z) &= \sup\big\{\la b+z,w\ra : w\in F^*\}\\
    &= \sup\big\{\big\la (z_1 + \alpha,z_2+\beta), (w_1,w_2)\big\ra : (w_1,w_2)\in F^*\big\}\\
    &= \sup\Big\{ (z_1+\alpha)w_1 + (z_2+\beta)w_2 :  w_1 \leq 1, i w_1+w_2 \leq 0, i=1,2, \ldots\Big\}.
\end{aligned}
\end{equation}}
As inspired by \cite[Example~3.1]{KhanhMoTrinh}, we observe that $z=(z_1,z_2)\in \dom v_{\mathcal{D}}$ if and only if $z_1+\alpha \ge z_2+\beta\ge 0$. Indeed, suppose that $z=(z_1,z_2)\in \dom v_{\mathcal{D}}$. As $(0,-k)$ and $(-k,k)$ belong to the {feasible set of the perturbed dual problem in \eqref{vDz}} for all $k\in \N$, we have 
\begin{equation*}
    -k (z_2+\beta) \le v_{\mathcal{D}}(z)\quad \text{and}\quad -k(z_1+\alpha) + k(z_2+\beta) \le v_{\mathcal{D}}(z),
\end{equation*}
and thus $z_2+\beta \ge -\dfrac{v_{\mathcal{D}}(z)}{k}$ and $(z_1+\alpha)-(z_2+\beta)\ge -\dfrac{v_{\mathcal{D}}(z)}{k}$ for all $k\in \N$. Letting $k\rightarrow \infty$ in the previous estimates yields $z_1+\alpha \ge z_2+\beta \ge 0$. Conversely, assume that $z_1+\alpha \ge z_2+\beta \ge 0$ and consider any $(w_1,w_2)\in \R^2$ for which 
\begin{equation}\label{w}
w_1\le 1 \ \text{ and }\ iw_1 + w_2\le 0\ \text{ for all }\ i=1,2,3,\ldots
\end{equation}
If $w_1>0$, then letting $i\rightarrow \infty$ in \eqref{w} leads to an absurd assertion, and hence we must have $w_1 \le 0$. Further, by employing \eqref{w} and $z_1+\alpha \ge z_2+\beta \ge 0$ one gets 
\begin{equation}\label{upper-bd}
    (z_1+\alpha)w_1 + (z_2+\beta)w_2 \le (z_2+\beta)(w_1+w_2) \le 0=(z_1+\alpha)\cdot 0 + (z_2+\beta)\cdot 0.
\end{equation}
Since the feasible set in \eqref{vDz} is nonempty and \eqref{upper-bd} holds for any feasible vector $(w_1,w_2)$, we have $v_{\mathcal{D}}(z)$ is finite, i.e., $z\in \dom v_{\mathcal{D}}$. Another observation stemming from the above clarification is that $v_{\mathcal{D}}(z)=0$ for all $z\in \dom v_{\mathcal{D}}$. As a consequence, the explicit representation of $\mathcal{N}$ is given as
\begin{equation*}
\begin{aligned}
    \mathcal{N}=\bigcup_{z\,\in\, \mathrm{dom}\, v_{\mathcal{D}}}\big(\{z\}\times [v_{\mathcal{D}}(z),+\infty[\big) &= \{(z_1,z_2,t_3)\in \R^3:z_1+\alpha \ge z_2+\beta \ge 0,t_3\ge 0\} \\
    &= \big\{(t_1+t_2-\alpha,t_2-\beta,t_3):t_1,t_2,t_3\ge 0\big\}.
\end{aligned}
\end{equation*}
\end{example}

{It is noteworthy that $\mathcal{N}$ and $\mathcal{M}$ are closely linked to $\mathcal{H}$ and $\mathcal{K}$, respectively.}

\begin{proposition}\label{prop:MN}
We have the following assertions:
\begin{enumerate}
    \item[{\rm (i)}] $\mathcal{H}\subset \mathcal{N}$ and $\mathcal{K}\subset \mathcal{M}$,
    \item[{\rm (ii)}] $\mathcal{N}$ and $\mathcal{M}$ are both convex,
    \item[{\rm (iii)}] If $\textup{val}(\mathcal{P})$ is finite, then $\big(0_Z,\textup{val}(\mathcal{P})\big)\in \mathcal{N}$.
    If $\textup{val}(\mathcal{D})$ is finite, then $\big(0_Y,\textup{val}(\mathcal{D})\big)\in \mathcal{M}$.
\end{enumerate}
\end{proposition}
\begin{proof}
We only need to verify the inclusion $\mathcal{H}\subset \mathcal{N}$ in (i) since the other inclusion is justified in the same manner. To this end, let us fix $(z,\alpha)\in \mathcal{H}$ and then find by the definition of $\mathcal{H}$ some vectors $x\in P,z\in Q$ such that $Ax-(b+z)=0_Z\in Q$ and $\la x,c\ra \le \alpha$. By using $\alpha \ge \la x,c\ra$ and by mimicking the proof of the weak duality result in Theorem \ref{theo:weak-dual}, we get immediately the estimates
\begin{equation*}
\begin{aligned}
    \alpha &\ge \inf\{\la x',c\ra :x'\in P,Ax'-(b+z)\in Q\}\\ &\ge \sup \{\la b+z,w\ra : -A^*w+c\in P^*,w\in Q^*\} = v_{\mathcal{D}}(z),
\end{aligned}
\end{equation*}
which obviously yield $(z,\alpha)\in \epi (v_{\mathcal{D}}) =  \mathcal{N}$.

The convexity claimed in (ii) is deduced directly from the fact that $\mathcal{N}=\epi (v_{\mathcal{D}})$ where $v_{\mathcal{D}}$ is convex, while $\mathcal{M}=\hypo (v_{\mathcal{P}})$ and the function $v_{\mathcal{P}}$ is concave. The assertions in (iii) are direct consequences of \eqref{N}-\eqref{M} and the weak duality in Theorem \ref{theo:weak-dual}. 
\end{proof}

The following example shows that the set $\mathcal{H}$ is strictly contained in $\mathcal{N}$.

\begin{example}\rm 
Let us examine the parametric linear problem considered in Example~\ref{exam:Gale1}.
By \cite[Example~3.1]{KhanhMoTrinh} and our own Example \ref{exam:Gale1}, we know that
\begin{equation*}
    \mathcal{H} = \left\{(t_1-\alpha,-\beta,t_1+t_2): t_1, t_2\ge 0\right\} \bigcup \left\{(t_1+t_2-\alpha,t_2-\beta,t_3):t_1,t_3\ge 0, t_2>0\right\},
\end{equation*}
and 
\begin{equation*}
    \mathcal{N} = \left\{(t_1+t_2-\alpha,t_2-\beta,t_3): t_1,t_2,t_3\ge 0\right\}.
\end{equation*}
Clearly, we have $(m-\alpha,-\beta,m/2) \in \mathcal{N}\setminus \mathcal{H}$ for any $m> 0$.
\end{example}

Below, we give a full characterization of the set $\mathcal{N}$ from \eqref{N} (resp. $\mathcal{M}$ from \eqref{M}) via information from $\mathcal{H}$ \eqref{Definition_H} (resp. $\mathcal{K}$ \eqref{Definition_K}). The obtained characterizations provide us with \textit{separation arguments} for points away from $\mathcal{N}$ (or $\mathcal{M}$), which shed light on the topological intimate bonds between $\mathcal{H}$ and $\mathcal{N}$ (resp. $\mathcal{K}$ and $\mathcal{M}$) presented later in Section \ref{sec:topo}. {We shall need the following lemma in subsequent results.}

{\begin{lemma}\label{auxili}
Let $(X,Y)$ be a dual pair of vector spaces and let $C$ be a nonempty strict cone in $X$, i.e., $\beta c\in C$ for all vector $c\in C$ and scalar $\beta>0$. For all $y\in Y$ and $\alpha \in \R$, we have
\begin{equation*}
    \big[\la x,y\ra \ge \alpha\,\, \forall x\in C\big] \Longrightarrow \big[y\in C^* \text{ and } \alpha \le 0\big].
\end{equation*}
\end{lemma}
\begin{proof}
Fix $y\in Y$ and $\alpha\in \R$ such that
\begin{equation}\label{22}
\la x,y\ra \ge \alpha \quad \forall x\in C.
\end{equation}
Pick any $x\in C$. Since $C$ is a strict cone, we get from \eqref{22} that $$\la \varepsilon x,y\ra \ge \alpha  \quad \text{for any} \quad \varepsilon>0.$$ 
By letting $\varepsilon \downarrow 0$, it follows that $\alpha \le 0$ as claimed. Regarding the other assertion, it suffices to show that $\la x,y\ra \ge 0$. Indeed, if $\la x,y\ra <0$, then one finds $M>0$ sufficiently large such that $Mx\in C$ and $\la Mx,y\ra < \alpha$, which clearly contradicts \eqref{22}. The proof is complete.
\end{proof}}

\begin{theorem}\label{contra} Consider the problems \eqref{P}-\eqref{D} and let $(z,r)\in Z\times \R$. The following assertions are equivalent:
\begin{enumerate}
\item[{\rm (i)}] $(z,r)\notin \mathcal{N}$,
\item[{\rm (ii)}] There exist $(w,\beta)\in W\times \R$ and $\gamma \in \R$ such that 
\begin{equation}\label{sepa-N}
    \big\la (z',r'),(w,\beta)\big\ra \le \gamma <\la (z,r),(w,\beta)\ra  \quad \forall (z',r')\in \mathcal{H},
\end{equation}
\item[{\rm (iii)}] There exists $(w,\beta)\in W\times \R$ such that 
\begin{equation}\label{equalemmaKM2'}
    \big\la (z',r')-(z,r),(w,\beta)\big\ra <0 \quad \forall (z',r')\in \mathcal{H}.
\end{equation}
\end{enumerate}
\end{theorem}
\begin{proof}
First we verify the implication [(i)$\Longrightarrow$(ii)] by assuming that $(z,r)\notin \mathcal{N}$. The latter implies that $r<v_{\mathcal{D}}(z)=\sup\{\la b+z,w\ra : -A^*w+c\in P^*,w\in Q^*\}$, which deduces the existence of $w\in Q^*$ satisfying 
\begin{equation}\label{prop-sepa-1}
-A^*w+c\in P^* \ \text{ and }\ r - \la b+z,w\ra<0. 
\end{equation}
Letting $\gamma:= -\la b,w\ra\in \R$, we see from \eqref{prop-sepa-1} that $ \gamma <\la (z,r),(w,-1)\ra$. On the other hand, fix $(z',r')\in \mathcal{H}$ and find by \eqref{Definition_H} vectors $x\in P,q\in Q$ such that $z'=Ax-b-q$ and $r'\ge \la x,c\ra$. Since $x\in P$, we deduce from \eqref{prop-sepa-1} and the definition of dual cones the inequality $\la x, -A^*w+c\ra \ge 0$, i.e., $\la Ax,-w\ra + \la x,c\ra \ge 0$. It follows that 
$$
\begin{aligned}
    \big\la (z',r'),(w,-1)\big\ra &= \la Ax-b-q,w\ra -r' \le -\la Ax,-w\ra -\la x,c\ra - \la q,w\ra +\gamma \\
    &\le -\la q,w\ra + \gamma \le \gamma
\end{aligned}
$$
due to $q\in Q$ and $w\in Q^*$. Hence, (ii) is justified and we know that [(i)$\Longrightarrow$(ii)] is true. The implication [(ii)$\Longrightarrow$(iii)] is trivial.

Lastly, we verify the implication [(iii)$\Longrightarrow$(i)] by supposing \eqref{equalemmaKM2'} holds for some $(w,\beta)\in W\times \R$ and supposing that $(z,r)\in \mathcal{N}=\epi (v_{\mathcal{D}})$. The latter means $r\ge v_{\mathcal{D}}(z)$ which amounts to 
\begin{equation}\label{equalemmaKM1'}
  r\ge \la b+z,w'\ra \quad \forall w'\in F(\mathcal{D}).  
\end{equation}
{We claim that \eqref{equalemmaKM2'} deduces
\begin{equation}\label{equalemmaKM2'''}
    \la x,-A^*w-\beta c\ra + \la q,w\ra +\alpha(-\beta) > -\la b+z,w\ra - \beta r
    \quad \forall (x,q,\alpha)\in P\times Q\times [0,+\infty[.
\end{equation}
Indeed, pick $(x,q,\alpha)\in P\times Q\times [0,+\infty[$ and see via \eqref{Definition_H} that $$\big(Ax-b-q,\la x,c\ra + \alpha\big)\in \mathcal{H}.$$
\eqref{equalemmaKM2'} then deduces
\begin{equation*}
    \big\la (Ax-b-q)-z,w\big\ra + \big( (\la x,c\ra +\alpha) -r\big)\beta <0,  
\end{equation*}
i.e., 
$$
\la x,-A^*w-\beta c\ra + \la q,w\ra +\alpha(-\beta) > -\la b+z,w\ra - \beta r.
$$
The assertion \eqref{equalemmaKM2'''} is thus justified.}

When $(x,q)=(0_X,0_Z)$, we may rewrite \eqref{equalemmaKM2'''} as
\begin{equation}\label{alla}
  \alpha (-\beta) > -\la b+z,w\ra - \beta r
    \quad \forall \alpha\in [0,+\infty[.
\end{equation}
In the case $(x,\alpha)=(0_X,0)$, \eqref{equalemmaKM2'''} becomes
\begin{equation}\label{allq}
    \la q,w\ra > -\la b+z,w\ra - \beta r \quad \forall q\in Q.
\end{equation}
Observe also that when $(q,\alpha)=(0_Z,0)$, \eqref{equalemmaKM2'''} can be rewritten as
\begin{equation}\label{allx}
\begin{aligned}
    \la x,-A^*w-\beta c\ra > -\la b+z,w\ra - \beta r
    \quad \forall x\in P.
\end{aligned}
\end{equation}
By employing consecutively Lemma \ref{auxili} for the cones $[0,+\infty[\,\subset \R$ in \eqref{alla} and $Q\subset Z$ in \eqref{allq} and $P\subset X$ in \eqref{allx}, one obtains that $\beta \le 0$ and $w\in Q^*$ and $-A^*w-\beta c\in P^*$.

{Take arbitrary $\gamma>0$ and $\widehat{w}\in F(\mathcal{D})$. Since $\big\{-A^*w -\beta c, -A^* \widehat{w}+c\big\}\subset P^*$, $\gamma>0$, and $\beta\le 0$, we have
\begin{equation*}
    -A^*\left(\dfrac{w+\widehat{w}\gamma}{-\beta+\gamma}\right)+c = \dfrac{1}{-\beta + \gamma} \big[(-A^*w-\beta c) + \gamma (-A^*\widehat{w}+c)\big] \in P^*.
\end{equation*}
Also, $\dfrac{w+\widehat{w}\gamma}{-\beta+\gamma}\in Q^*$ as $w,\widehat{w}\in Q^*$. Having $\dfrac{w+\widehat{w}\gamma}{-\beta+\gamma}\in F(\mathcal{D})$, we get from \eqref{equalemmaKM1'} the estimate
\begin{equation*}
    r\ge \Big\la b+z, \dfrac{w+\widehat{w}\gamma}{-\beta+\gamma}\Big\ra,
\end{equation*}
i.e., $r(-\beta + \gamma) \ge \la b+z, w + \widehat{w}\gamma\ra$. Letting $\gamma \downarrow 0$, one arrives at 
\begin{equation*}
    -r\beta - \la b+z,w\ra \ge 0
\end{equation*}
which clearly contradicts \eqref{allq} at $q:=0_Z$.} As a consequence, we have $(z,r)\notin \mathcal{N}$ as claimed.
\end{proof}

Similarly in $Y\times \R$, one obtains the following ``separation-type" characterization for $\mathcal{M}$.

\begin{theorem}\label{contra-KM} Consider the problems \eqref{P}-\eqref{D} and let $(y,r)\in Y\times \R$. Consider the following assertions:
\begin{enumerate}
\item[{\rm (i)}] $(y,r)\notin \mathcal{M}$,
\item[{\rm (ii)}] There exist $(x,\beta)\in X\times \R$ and $\gamma\in \R$ such that 
\begin{equation*}
    \big\la (x,\beta),(y',r')\big\ra \le  \gamma <\la (x,\beta),(y,r)\ra \quad \forall (y',r')\in \mathcal{K}.
\end{equation*}
\item[{\rm (iii)}] There exists $(x,\beta)\in X\times \R$ such that 
\begin{equation*}
    \big\la (x,\beta),(y',r')-(y,r)\big\ra <0 \quad \forall (y',r')\in \mathcal{K}.
\end{equation*}
\end{enumerate}
Then {\rm [(i)$\Longrightarrow$(ii)]} and {\rm [(ii)$\Longrightarrow$(iii)]} hold in general, while the implication {\rm [(iii)$\Longrightarrow$(i)]} is true when $P^{**}=P$, where $P^{**}$ is given in \eqref{P**}. 
\end{theorem}

Next, we establish some remarkable relations between $\mathcal{N}$ from \eqref{N} and the algebraic closure of $\mathcal{H}$ in $Z\times \R$. First, let us recall from Holmes \cite[p.~9]{Holmes} the set of \textit{linearly accessible points} of $\Omega \subset X$ as 
\begin{equation*}
    \lina \Omega :=\{x\in X: \exists\, \ox\in \Omega \setminus \{x\} \text{ such that } [\ox,x[\, \subset\, \Omega \},
\end{equation*}
where $[\ox,x[:=\big\{(1-\lambda)\ox + \lambda x: \lambda \in [0,1[\big\}$. Then, the \textit{algebraic closure} of $\Omega$ is 
\begin{equation*}
    \acl \Omega := \Omega \cup \lina \Omega.
\end{equation*}
Notably, the algebraic closure of a convex set $\Omega\subset X$ coincides with its (topological) closure provided that either $X$ is finite-dimensional \cite[Section~2D,~p.~9]{Holmes} or $X$ is merely a topological vector space and the interior of $\Omega$ is nonempty \cite[Lemma,~p.~59]{Holmes}.

Surprisingly, Proposition \ref{radiallycont} and Proposition \ref{prop:MN} ensure that $\mathcal{N}$ contains the algebraic closure $\acl \mathcal{H}$ in $Z\times \R$. 
\begin{proposition}\label{aclHandM}
For $\mathcal{H}$ and $\mathcal{N}$ given, respectively, by \eqref{Definition_H} and \eqref{N}, we have $$\acl \mathcal{H} \subset  \mathcal{N}.$$   
\end{proposition}
\begin{proof}
As $\mathcal{H}\subset \mathcal{N}$ by Proposition \ref{prop:MN}, showing that $\lina \mathcal{H}\subset \mathcal{N}$ is sufficient. To this end, suppose there is $(z,r)\in \lina \mathcal{H}$, i.e., there exists $(\oz,\overline{r})\in \mathcal{H} \setminus \{(z,r)\}$ satisfying $[(\oz,\ovr),(z,r)[ \subset \mathcal{H} \subset \mathcal{N}$, and thus 
\begin{equation}\label{sub-1}
v_{\mathcal{D}}\big((1-\lambda)z + \lambda \oz\big) \le (1-\lambda)r + \lambda \ovr \ \text{ for all }\ \lambda\in\, ]0,1].
\end{equation}
Besides, recall from Proposition \ref{radiallycont} that $v_{\mathcal{D}}$ is radially lower semicontinuous, which together with \eqref{sub-1} deduces the relationships
\begin{equation*}
v_{\mathcal{D}}(z) \le \liminf_{\lambda\downarrow 0} v_{\mathcal{D}}\big((1-\lambda)z + \lambda \oz\big) \le \liminf_{\lambda\downarrow 0} \big( (1-\lambda)r + \lambda \ovr\big) = r,
\end{equation*}
which thus confirm $(z,r)\in \mathcal{N}$. The inclusion $\acl \mathcal{H}\subset \mathcal{N}$ is hence justified.
\end{proof}

Motivated by the \textit{closedness conditions} (e.g., \cite{Kretschmer61,KhanhMoTrinh}), we state the following conditions which play a central role in our subsequent study.
\begin{equation*}
\begin{split}
&\text{{\bf Condition (D)}: $\quad\mathcal{H}\, =\,\mathcal{N}$,}\\
&\text{{\bf Condition (D$^*$)}: $\quad\mathcal{K}\,=\,\mathcal{M}$.}
\end{split}
\end{equation*}

Loosely speaking, we shall elaborate in Subsection \ref{subsub:strongdual} that under a feasibility condition, condition {\bf (D)} amounts to a \textit{perturbed strong duality-type} result. In Section \ref{sec:algebra}, we provide an algebraic sufficient condition for {\bf (D)} via the \textit{algebraic core} of the cone $P^*$ in question. On the other hand, in Section \ref{sec:topo}, we confirm that condition {\bf (D)} reduces exactly to the closedness of $\mathcal{H}$ from \eqref{Definition_H}, when topological structures are incorporated. The analogous results involving {\bf (D$^*$)} follow by the same reasoning.

To establish the desired perturbed strong duality results, first we need to construct some corresponding versions of perturbed Farkas lemmas as follows. 
\section{Farkas lemma and strong duality}\label{sec:strongdual}
\subsection{Algebraic perturbed Farkas lemma}\label{Farkas}
This subsection is devoted to the study of two versions of perturbed Farkas lemma for infinite-dimensional conic linear programming. First, we have the following perturbed form of Farkas lemma. To the best of our knowledge, results in this manner appeared for the first time in \cite[Theorem~3]{Zalinescu78}. 

\begin{theorem}[Perturbed Farkas lemma 1]\label{theo:fixz}
Consider the problems \eqref{P}-\eqref{D} and fix $z\in Z$. The following statements are equivalent:
\begin{enumerate}
    \item[{\rm (i)}] $\mathcal{H}\cap (\{z\}\times \R) = \mathcal{N}\cap (\{z\}\times \R)$,
    \item[{\rm (ii)}] For any $\alpha \in \R$, the following assertions are equivalent:
    \begin{enumerate}
    \item[{\rm (a)}] $\forall w\in Q^* \text{ satisfying } -A^*w+c\in P^*$, we have $\langle b+z, w\rangle \leq \alpha$. \item[{\rm (b)}] $\exists\, x\in P\ \text{such\ that}\ Ax-b-z\in Q\text{ and }\la x,c\ra \le \alpha$.%
    \end{enumerate}
\end{enumerate}
\end{theorem}
\begin{proof}
[(i)$\Longrightarrow$(ii)] Assume that (i) holds. Fix $\alpha\in \R$. We first prove that [(a)$\Longrightarrow$(b)] in (ii) is true. 
Indeed, one has 
\[\alpha \ge \langle b+z,w\rangle, \ \forall w \in Q^* \text{ satisfying }-A^*w+c\in P^*\]
is equivalent to $\alpha \geq v_{\mathcal{D}}(z)$, which means that 
$$(z,\alpha) \in \epi ( v_{\mathcal{D}})\cap (\{z\}\times \R) = \mathcal{N}\cap (\{z\}\times \R)=\mathcal{H}\cap (\{z\}\times \R)$$ 
by \eqref{N} and (i). By the definition of $\mathcal{H}$,  the inclusion ensures the existence of $x\in P$ such that $Ax-b-z\in Q$ and $\la x,c\ra \le \alpha$ which verifies [(a)$\Longrightarrow$(b)] in (ii). We then show that the implication [(b)$\Longrightarrow$(a)] of (ii) holds by assuming the existence of $x\in P$ such that 
\begin{equation}\label{eq:up-1}
Ax-b-z\in Q\text{ and }\la x,c\ra \le \alpha.
\end{equation}
Pick any $w\in Q^*$ satisfying $-A^*w + c \in P^*$. Then by $x\in P$, $Ax-b-z\in Q$, and the definition of positive dual cones, we have $\la x,-A^*w+c\ra \ge 0$ and $\la Ax-b-z,w\ra\ge 0$. Combining this with \eqref{eq:up-1} yields
\begin{equation*}
    \la b+z,w\ra \le \la Ax,w\ra = \la x,A^*w\ra \le \la x,c\ra \le \alpha.
\end{equation*}
Observe that the implication [(b)$\Longrightarrow$(a)] remains valid even without assumption (i).

[(ii)$\Longrightarrow$(i)] Assume that (ii) holds. We only need to show that $\mathcal{N}\cap (\{z\}\times \R) \subset \mathcal{H}\cap (\{z\}\times \R)$. To do this, take $\alpha \in \R$ for which $(z,\alpha) \in \mathcal{N}$. Therefore, $$(z,\alpha) \in \bigcup_{z\,\in\, \limfunc{dom}v_{\mathcal{D}}}\left(\{z\}\times[ v_{\mathcal{D}}(z), +\infty[ \right)$$ by \eqref{N}, and thus $\la b+z,w\ra \le \alpha$ for all $w\in F(\mathcal{D})$. As (ii) holds, it follows that there exists $x\in P$ satisfying
\begin{equation*}
    Ax-b-z\in Q\ \text{ and }\ \la x,c\ra \le \alpha,
\end{equation*}
i.e., $(z,\alpha)\in \mathcal{H}$. The proof is complete.
\end{proof}

\begin{remark}
Comparing to \cite[Theorem~3]{Zalinescu78}, Theorem \ref{theo:fixz} provides not only a sufficient but also necessary condition for the equivalence \textup{[(a)$\Longleftrightarrow$(b)]} via the identity $\mathcal{H}=\mathcal{N}$ which is shown later in Section \ref{sec:topo} to be equivalent to the weak closedness of $\mathcal{H}$ in $Z\times \R$.
\end{remark}

Letting $z:=0_Z$ in Theorem \ref{theo:fixz}, we obtain a generalization of a well-known form of Farkas lemma for finite-dimensional linear programming.

\begin{theorem} [Farkas lemma 1]\label{Farkas1ILP}  Regarding the problems \eqref{P}-\eqref{D}, the following statements are equivalent:
\begin{enumerate}
    \item[{\rm (i)}] $\mathcal{H}\cap (\{0_Z\}\times \R) = \mathcal{N}\cap (\{0_Z\}\times \R)$,
    \item[{\rm (ii)}] For any $\alpha \in \R$, the following assertions are equivalent:
    \begin{enumerate}
    \item[{\rm (a)}] $\forall w\in Q^* \text{ and }\ -A^*w+c\in P^*$, we have $\langle b, w\rangle \leq \alpha$. \item[{\rm (b)}] $\exists\, x\in P\ \text{such\ that}\ Ax-b\in Q\text{ and }\la x,c\ra \le \alpha$.%
    \end{enumerate}
\end{enumerate}
\end{theorem}

We now apply these arguments again, with $\mathcal{H}$ replaced by $\mathcal{K}$, to obtain the dual counterparts to above.

\begin{theorem}[Perturbed Farkas lemma 2]\label{theo:fixy}
Consider the problems \eqref{P}-\eqref{D} and fix $y\in Y$. The following statements are equivalent:
\begin{enumerate}
    \item[{\rm (i)}] $\mathcal{K}\cap (\{y\}\times \R) = \mathcal{M}\cap (\{y\}\times \R)$,
    \item[{\rm (ii)}] For any $\alpha \in \R$, the following assertions are equivalent:
    \begin{enumerate}
    \item[{\rm (a)}] $\forall x\in P,\ Ax-b\in Q$, we have $\langle x,c-y \rangle \geq \alpha$. 
    \item[{\rm (b)}] $\exists\, w\in Q^*$ such that $-A^*w+c-y\in P^*$ and $\la b,w\ra \ge \alpha$.%
    \end{enumerate}
\end{enumerate}
\end{theorem}

Specifically when $y:=0_Y$, we have

\begin{theorem} [Farkas lemma 2]\label{Farkas2ILP} Regarding the problems \eqref{P}-\eqref{D}, the following statements are equivalent:
\begin{enumerate}
    \item[{\rm (i)}] $\mathcal{K}\cap (\{0_Y\}\times \R) = \mathcal{M}\cap (\{0_Y\}\times \R)$,
    \item[{\rm (ii)}] For any $\alpha \in \R$, the following assertions are equivalent:
    \begin{enumerate}
    \item[{\rm (a)}] $\forall x\in P \text{ such that } Ax-b\in Q$, we have $\langle x,c \rangle \geq \alpha$. 
    \item[{\rm (b)}] $\exists\, w\in Q^*$ such that $-A^*w+c\in P^*$ and $\la b,w\ra \ge \alpha$.%
    \end{enumerate}
\end{enumerate}
\end{theorem}

\subsection{Algebraic perturbed strong duality}\label{subsub:strongdual}

Here in this subsection we show that the Farkas lemmas for conic linear programming in Section \ref{Farkas} can be used to establish necessary and sufficient conditions for the  \textit{perturbed strong duality}. Next, we present to the readers our first perturbed form of strong duality.

\begin{theorem}[Perturbed strong duality 1]\label{strongduality1-z} Consider the problems \eqref{P}-\eqref{D} and fix $z\in Z$. The following statements are equivalent:
	\begin{enumerate}
		\item[{\rm (i)}]$\emptyset\neq \mathcal{H}\cap (\{z\}\times \R)=\mathcal{N}\cap(\{z\}\times\R)$,
		\item[{\rm (ii)}] $\min\{\langle x,c \rangle:Ax-b-z\in Q, x\in P\} = \sup\{\langle b+z, w\rangle:-A^*w+c \in P^*, w\in Q^*\}.$
	\end{enumerate} 
\end{theorem}
\begin{proof}
First we justify the implication [(i)$\Longrightarrow$(ii)] by assuming the validation of (i). Above all, we show that the infimum value $\inf\{\la x,c\ra:Ax-b-z\in Q,x\in P\}$ is finite and is attained. By the definition of the dual cones $P^*,Q^*$ and the asssumption $F(\mathcal{D})\neq \emptyset$, it is obvious that
\begin{equation}\label{perturbeddualityequa1}
	\inf\{\langle x,c \rangle: Ax-b-z\in Q, x\in P\} \geq \sup\{\langle b+z,w\rangle:-A^*w+c\in P^*, w\in Q^*\}>-\infty.
\end{equation}
On the other hand, the nonemptiness of $\mathcal{H}\cap (\{z\}\times \R)$ lends us a primal feasible solution $x\in P$ at which $Ax-b-z\in Q$. Combining this with \eqref{perturbeddualityequa1} ensures that the infimum in \eqref{perturbeddualityequa1} is finite. Employing now \eqref{N}, \eqref{perturbeddualityequa1} and assumption (i) yields $${\Big(z,\inf\{\la x,c\ra : Ax-b-z\in Q,x\in P\}\Big)\in \mathcal{N}\cap (\{z\}\times \R) =\mathcal{H}\cap (\{z\}\times \R).}$$
Consequently, by \eqref{Definition_H} there exists $\ox\in P$ such that $A\ox-b-z\in Q$ and $$\la \ox,c\ra \le \inf\{\la x,c\ra : Ax-b-z\in Q,x\in P\},$$ i.e., the infimum in \eqref{perturbeddualityequa1} is attained at $\ox$.

Now we verify the identity in (ii). By \eqref{perturbeddualityequa1} and the above reasoning, we have $${\sup\{\langle b+z,w\rangle:-A^*w+c\in P^*, w\in Q^*\} = \alpha \in \R.}$$
By the supremum representation of $\alpha$, we can write $\langle b+z,w \rangle \leq \alpha$ for all $w \in F(\mathcal{D})$. Combining the latter with assertion (i) and {the perturbed Farkas lemma 1 (Theorem \ref{theo:fixz})}, one finds $\overline{x}\in P$ such that $A\overline{x}-b-z\in Q$ and $\la \overline{x},c\ra \le \alpha$, and thus
\begin{equation*}
    \min\{\langle x,c \rangle: Ax-b-z\in Q, x\in P\} \le \la\overline{x},c\ra \le \alpha
\end{equation*}
which justifies the equality in (ii).

Now we suppose (ii) holds and claim (i). Since we have
\begin{equation*}
    \inf\{\la x,c\ra : Ax-b-z\in Q,x\in P\} = \min \{\la x,c\ra : Ax-b-z\in Q,x\in P\} <\infty,
\end{equation*}
the feasible set $\{x\in P:Ax-b-z\in Q\}$ is nonempty and so is the set $\mathcal{H}\cap (\{z\}\times \R)$. Now observe via Theorem \ref{theo:fixz} that the identity in (i) is equivalent to the validation of [(a)$\Longleftrightarrow$(b)] for each $\alpha \in \R$ in Theorem~\ref{theo:fixz}. This tells us to fix any $\alpha \in \R$, and establish via assumption (ii) the equivalences
\begin{equation*}
\begin{aligned}
    &\la b+z,w \ra \le \alpha \quad \forall w\in Q^*,-A^*w+c\in P^* \\
    \Longleftrightarrow &\;\sup\{\langle b+z, w\rangle:w\in Q^*,-A^*w+c \in P^*\} \le \alpha \\
    \Longleftrightarrow &\;\min\{\langle x,c \rangle:Ax-b-z \in Q, x\in P\}\le \alpha \\
    \Longleftrightarrow &\;\exists\, x\in P,Ax-b-z\in Q:\; \la x,c\ra \le \alpha.
\end{aligned}
\end{equation*}
As [(a)$\Longleftrightarrow$(b)] in Theorem \ref{theo:fixz} is fulfilled for each $\alpha\in \R$, it follows from such theorem that the identity in (i) holds.
\end{proof}

As a consequence of Theorem \ref{strongduality1-z}, we obtain the following strong duality result which serves as an algebraic counterpart to Khanh et al. \cite[Theorem~4.4]{KhanhMoTrinh}. 

\begin{theorem} {\rm {\bf (Strong duality 1)}} \label{strongduality1}
Regarding the problems \eqref{P}-\eqref{D}, the following statements are equivalent:
	\begin{enumerate}
		\item[{\rm (i)}]$\mathcal{H}\cap (\{0_Z\}\times \R)=\mathcal{N}\cap(\{0_Z\}\times\R)$,
		\item[{\rm (ii)}] $\min\{\langle x,c \rangle:Ax-b \in Q, x\in P\} = \sup\{\langle b, w\rangle:-A^*w+c \in P^*, w\in Q^*\}.$
	\end{enumerate} 
\end{theorem}

Again, we replace $\mathcal{H}$ by $\mathcal{K}$ to acquire the second perturbed and unperturbed strong duality results. 

\begin{theorem}[Perturbed strong duality 2]\label{strongduality2-y} Consider the problems \eqref{P}-\eqref{D} and fix $y\in Y$.
The following statements are equivalent:
	\begin{enumerate}
		\item[{\rm (i)}]$\emptyset \neq \mathcal{K}\cap (\{y\}\times \R)=\mathcal{M}\cap(\{y\}\times\R)$,
		\item[{\rm (ii)}] $\inf\{\langle x,c - y\rangle:Ax-b\in Q, x\in P\} = \max\{\langle b,w\rangle:-A^*w+c - y\in P^*, w\in Q^*\}$.
	\end{enumerate} 
\end{theorem}

\begin{theorem} {\rm {\bf (Strong duality 2, cf. \cite[Corollary~4.2]{KhanhMoTrinh})}} \label{strongduality2}Regarding the problems \eqref{P}-\eqref{D}, the following statements are equivalent:
	\begin{enumerate}
		\item[{\rm (i)}]$\mathcal{K}\cap (\{0_Y\}\times \R)=\mathcal{M}\cap(\{0_Y\}\times\R)$,
		\item[{\rm (ii)}]$\inf\{\langle x,c \rangle:Ax-b\in Q, x\in P\} = \max\{\langle b,w\rangle:-A^*w+c \in P^*, w\in Q^*\}.$
	\end{enumerate} 
\end{theorem}
 
\section{Algebraic sufficient conditions for (D) and (D$^*$)}\label{sec:algebra}
In this section, let $X$ and $Z$ denote real vector spaces.  Let us consider next the \textit{algebraic duals} $$
{X^*}:=\{f:X\rightarrow \R:f\text{ is linear}\}
$$
of $X$ and ${Z^*}$ of $Z$. Note that $(X,X^*)$ and $(Z,Z^*)$ are dual pairs of vector spaces as defined in Section \ref{sec:prelim}. Let $A$ be a linear map acting from $X$ to $Z$. The \textit{adjoint} $A^*:Z^*\rightarrow X^*$ of $A$ is a linear map defined using the relation 
\begin{equation*}
    \langle x, A^*z^*\rangle=\langle Ax,z^*\rangle\quad \forall x\in X,\;\forall z^*\in Z^*.
\end{equation*}
We consider here the problems $(\mathcal{P})$ and $(\mathcal{D})$ from \eqref{P}-\eqref{D}, only replacing $Y$ by $X^*$ and $W$ by $Z^*$ for later technical purposes. 

The \textit{subspace spanned by $S\subset X$}, denoted by $\spa S$, is the smallest subspace of $X$ containing $S$. Given a subset $S\subset X$, the \textit{algebraic core} \cite[p.~7]{Holmes} of $S$ in $X$, denoted by $\cor S$, is defined as
\begin{equation*}
    \cor S:=\{x\in X: \forall y\in X,\; \exists\, t_0> 0 \text{ such that } x+ty\in S \text{ whenever } 0\le t\le t_0\}.
\end{equation*}
We may employ instead an equivalent representation:
\begin{equation}\label{eq:core}
    \cor S:=\{x\in X: \forall y\in X,\; \exists\, t_0> 0 \text{ such that } x+ty\in S \text{ whenever } |t|\le t_0\}.
\end{equation}
One easily checks that
\begin{equation*}
    \cor (-S) = - \cor S.
\end{equation*}

\begin{lemma}\label{lem:core} In a product space $X\times Y$, $\cor (A\times B) =  \cor A \times \cor B$.
\end{lemma}
\begin{proof}
We fix $(a,b)\in \cor (A \times B) \subset X\times Y$ and claim that $a\in \cor A$ and $b\in \cor B$. To this end, let $x\in X$ and $y\in Y$ be arbitrarily. As $(x,y)\in X\times Y$ and $(a,b)\in \cor (A\times B)$, by \eqref{eq:core} there exists $t_0>0$ such that 
\begin{equation*}
    (a,b)+t(x,y)\in A\times B \ \text{ whenever }\ |t|\le t_0,
\end{equation*}
i.e., $a+tx\in A$ and $b+ty\in B$ for all $t \in [-t_0,t_0]$. The latter assertion implies that $a\in \cor A$ and $b\in \cor B$ which justify the inclusion $(a,b)\in \cor A \times \cor B$. The opposite inclusion is verified in the same manner.
\end{proof}

According to \cite[p.~15,~\S4]{Holmes}, a linear functional $f\in X^*$ is said to \textit{separate} $A,B\subset X$ if there exists $\alpha \in \R$ such that we have either $A\subset \{x:f(x) \leq \alpha\}$ and $B\subset \{x:f(x) \geq \alpha\}$ or $A\subset \{x:f(x)\geq \alpha\}$ and $B\subset \{x:f(x)\leq \alpha\}$. The following \textit{point-set separation} argument plays an essential role in our main result later.
\begin{theorem}{\rm {\bf (see \cite[Theorem~2.1.3]{BotGradWanka} or \cite[Corollary~4B]{Holmes})}}\label{theo:sepacore} Let $E$ and $F$ be nonempty convex subsets of a vector space $X$ with $ \cor E \neq \emptyset$. Then $E$ and $F$ can be separated by a nonzero linear functional on $X$ if and only if $F \cap \cor E = \emptyset$.	
\end{theorem}

Fix $z \in \dom v_{\mathcal{D}}$ and define a perturbation of $\mathcal{K}$ in \eqref{Definition_K} as 
\begin{eqnarray}\label{set-T}
	{\mathcal{K}_z :=\displaystyle\bigcup_{w\,\in\, Q^*}\big((-A^*w+c-P^*)\,\times\, ]-\infty,\langle b+z,w\rangle]\big).}
\end{eqnarray}

\begin{proposition}\label{prop:properties-set-T}
Let $z \in \dom v_{\mathcal{D}}$. Then we have the following assertions:
   \begin{enumerate}
   	\item[{\rm (i)}] $\mathcal{K}_z$ is a nonempty convex subset of $X^*\times \R$ and there exists $\bar{w}\in Q^*$ such that
    \begin{equation}\label{eq:coreKz}
    	{(-\cor P^*)} \times \big ]- \infty, \langle b + z, \bar w \rangle\big[ \subset \cor \K_z,
    \end{equation}
   	\item[{\rm (ii)}] $(0_{X^*}, v_{\mathcal{D}}(z)) \notin \cor \K_z$,
   	\item[{\rm (iii)}] For any $\epsilon>0$, we have  $(0_{X^*}, v_{\mathcal{D}}(z) - \epsilon) \in  \K_z$,
   	\item[{\rm (iv)}] {Suppose that 
\begin{equation}\label{condition-core-Pstar}
\exists\,   \bar w \in Q^* : - A^*\bar w + c \in \cor P^*. 
\end{equation}}
Then for any $y \in X^*$ there exists $\delta_0 >0$  such that  $\left(\delta y, \langle b + z, \bar w \rangle\right) \in  \K_z$ whenever $|\delta|\leq\delta_0$.
   	\end{enumerate}    
\end{proposition}

\begin{proof}
First we show that {\rm (i)} is true. Indeed, since $F(\mathcal{D})$ is nonempty, there exists $ \bar w \in Q^*$ such that $- A^*\bar w + c \in P^*$, i.e., $0_{X^*} \in  -A^*\bar w + c -  P^*$. Hence, $\big(0_{X^*}, \langle b + z, \bar w \rangle\big) \in \K_z$ and $\K_z$ is nonempty. The convexity of $\K_z$ follows from the fact that 
\begin{equation*}
    \K_z = L_z (Q^*\times P^*\times \R_+) + \{(c,0)\},
\end{equation*}
in which $L_z: Z^*\times X^* \times \R \to X^* \times \R$ is given by
\begin{equation*}
    L_z (w,p,\alpha):= (-A^*w -p, \la b+z,w\ra - \alpha),\quad \forall (w,p,\alpha)\in Z^*\times X^* \times \R,
\end{equation*}
is a linear map.

We next verify the inclusion \eqref{eq:coreKz}. Take arbitrarily $y \in P^*$. Since $P^*$ is a convex cone and $0_{X^*} \in  -A^*\bar w + c - P^*$, we have  
\begin{equation}\label{prop:properties-set-T-equa3}
-y \in  -A^* \bar w + c - y -  P^* \subset -A^* \bar w +  c -  P^*.
\end{equation}
This gives rise to $(-y, \alpha) \in \K_z$ for all $\alpha \leq \langle b + z, \bar w \rangle$. 
Further, this inclusion holds for any $y \in P^*$.
Hence, 
\begin{equation*}
	-P^* \times \big]- \infty, \langle b + z, \bar w \rangle\big] \subset \K_z.
\end{equation*}
According to this and Lemma \ref{lem:core}, one has
\begin{equation*}
	(-\cor P^*) \times \big] - \infty, \langle b + z, \bar w \rangle\big[ = \cor \big(-P^* \, \times\, ] - \infty, \langle b + z, \bar w \rangle\big]\big) \subset \cor \K_z.
\end{equation*} 

To verify {\rm (ii)}, suppose contrary to our claim that $(0_{X^*}, v_{\mathcal{D}}(z)) \in \cor \K_z$. Then by the definition of $\cor \K_z$, we can find $\epsilon > 0$ such that 
\begin{equation*}
(0_{X^*}, v_{\mathcal{D}}(z)) +	\epsilon \big(0_{X^*}, 1\big) \in \K_z.
\end{equation*}   
Then \eqref{set-T} lends us some $\widetilde{w} \in Q^*$ such that $- A^* \widetilde{w} +c \in  P^* $ and $\langle b + z, \widetilde{w} \rangle \geq v_{\mathcal{D}}(z) + \epsilon,$ which clearly contradicts the definition of $v_{\mathcal{D}}(z)$ in \eqref{Dy}.
 
We proceed to show that {\rm (iii)} holds. Indeed, for any $\e>0$, the definition of $v_{\mathcal{D}}$ from \eqref{Dy} lends us some $w\in Q^*$ such that $0_{X^*}\in -A^*w + c - P^*$ and $\la b+z,w\ra > v_{\mathcal{D}}(z)-\e$. And thus, 
$(0_{X^*}, v_{\mathcal{D}}(z) - \epsilon) \in  \K_z$.

Lastly, we verify {\rm (iv)} by taking an arbitrary $y \in X^*$. Using \eqref{condition-core-Pstar}, we find $\delta_0>0$ such that whenever $|\delta|\leq\delta_0$, there holds the inclusion  
$$
	 - A^*\bar w + c - \delta y\in  P^*,
$$
i.e.,
$$
	\delta y \in  -A^* \bar w +  c -  P^*.
$$
Consequently, {\rm (iv)} is true. 
\end{proof}

The \textit{natural embedding} from a vector space $X$ to its algebraic bidual $X^{**}$ is defined as $i_X:X\to X^{**}$, where  
\begin{equation*}
    \la x^*,i_X(x) \ra:= \la x,x^*\ra \quad \forall x^* \in X^*
\end{equation*}
for each $x\in X$.

\begin{theorem}\label{theo:main-algebra}
If condition (\ref{condition-core-Pstar}) holds and $(P^{*})^*=i_X (P)$, then $\mathcal{H}=\mathcal{N}$.
\end{theorem}
\begin{proof}
As the inclusion $\mathcal{H}\subset \mathcal{N}$ follows from Proposition \ref{prop:MN}(i), it is left to show that $\mathcal{N} \subset \mathcal{H}$. We proceed to take arbitrarily $(z,r)\in \mathcal{N}$. By \eqref{N}, one has $z \in \dom v_{\mathcal{D}}$ and $r \geq v_{\mathcal{D}}(z)$. Proposition~\ref{prop:properties-set-T}(i) states that the set $\K_z$ from \eqref{set-T} is nonempty and convex with $\cor \K_z \neq \emptyset$. Moreover, by Proposition~\ref{prop:properties-set-T}(ii), one has $(0_{X^*}, v_{\mathcal{D}}(z)) \notin \cor \K_z$. We now apply the point-set separation argument from Theorem~\ref{theo:sepacore} to obtain
 the existence of a nonzero vector $(\bar x, \bar \alpha) \in X^{**}\times \R$ such that 
\begin{equation}\label{theo:main-algebra-equa1}
	\bar \alpha v_{\mathcal{D}}(z) \geq \langle y,\bar x \rangle + \bar \alpha \alpha \ \text{ for all }\ (y, \alpha) \in \K_z.
\end{equation} 
According to Proposition \ref{prop:properties-set-T}(iii), there exists $\epsilon >0$ such that $(0_{X^*}, v_{\mathcal{D}}(z) - \epsilon) \in  \K_z$. 
By substituting $(0_{X^*}, v_{\mathcal{D}}(z) - \epsilon)\in \K_z$ into    (\ref{theo:main-algebra-equa1}), we have
\begin{equation*}
	\bar \alpha v_{\mathcal{D}}(z) \geq  \bar \alpha \left(v_{\mathcal{D}}(z) - \epsilon\right),
\end{equation*}  
which implies that $\bar \alpha \geq 0$. If $\bar \alpha = 0$, then by substituting $\bar \alpha = 0$ into (\ref{theo:main-algebra-equa1}) we get \begin{equation}\label{theo:main-algebra-equa2}
 	0 \geq \langle y,\bar x \rangle \ \text{ for all }\ (y, \alpha) \in \K_z.
\end{equation}
In fact, \eqref{theo:main-algebra-equa2} leads to $\ox =0_{X^{**}}$. To see this, pick any $y' \in X^*$. Thanks to Proposition~\ref{prop:properties-set-T}(iv), there exists $\delta >0$  such that $$\left(\delta y', \langle b + z, \bar w \rangle\right) \in  \K_z\quad \text{and}\quad \left(- \delta y', \langle b + z, \bar w \rangle\right) \in \K_z.$$ This together with (\ref{theo:main-algebra-equa2}) gives rise to $ \langle y',\bar x \rangle = 0$. Since this relationship is true for any $y' \in X^*$, it follows that $ \bar x = 0_{X^{**}}$, contrary to $(\bar x, \bar \alpha) \neq \left(0_{X^{**}}, 0\right)$. 
Hence, $\bar \alpha > 0$ and we may assume without loss of generality that $\bar \alpha = 1$. Now (\ref{theo:main-algebra-equa1}) is equivalent to
\begin{equation}\label{theo:main-algebra-equa3}
	 v_{\mathcal{D}}(z) \geq  \langle y,\bar x \rangle + \alpha \ \text{ for all }\  (y, \alpha) \in \K_z.
\end{equation}

We proceed to show that $\bar x \in (P^{*})^*$ by fixing $y \in P^*$, $w \in Q^*$. Take arbitrarily $\lambda,\eta>0$. As $\lambda y \in P^*$, we have $$ -A^*(\eta w)  + c - \lambda y\in -A^*(\eta w)  + c - P^*,$$
which deduces $\big(-A^*(\eta w)  + c - \lambda y, \langle b + z,  \eta w \rangle\big) \in \K_z$. Combining this with (\ref{theo:main-algebra-equa3}), we have 
\begin{equation}\label{theo:main-algebra-equa4}
	 v_{\mathcal{D}}(z) \geq   \langle -A^*(\eta w)  + c - \lambda y,\bar x \rangle + \langle b + z,  \eta w \rangle \ \text{ for all }\lambda,\eta>0.	
\end{equation} 
Specifically at $\eta=1$, \eqref{theo:main-algebra-equa4} yields
\begin{equation*}
	\dfrac{1}{\lambda} v_{\mathcal{D}}(z) \geq   \dfrac{1}{\lambda} \langle  -A^*w  + c,\bar x \rangle -\langle y,\bar x \rangle + \dfrac{1}{\lambda} \langle b + z,  w \rangle \ \text{ for all }\ \lambda >0.
\end{equation*} 
Letting $\lambda \longrightarrow \infty$, we get $\langle y,\bar x \rangle \geq 0$, which holds for any $ y \in P^*$. Thus $\bar x \in (P^{*})^* = i_X(P)$. Consequently, $\ox = i_X(\tilde{x})$ for some $\tilde{x}\in P$.

On the other hand, \eqref{theo:main-algebra-equa4} asserts in the case $\lambda=1$ that
\begin{equation*}
	\dfrac{1}{\eta} v_{\mathcal{D}}(z) \geq  \langle - A\tilde{x} + b + z,  w \rangle   + \dfrac{1}{\eta}\langle c - y,\bar x \rangle \ \text{ for all }\ \eta >0.
\end{equation*} 
By letting $\eta \longrightarrow \infty$, one gets $\langle - A\tilde{x} + b + z,  w \rangle \leq 0$ for arbitrary $w\in Q^*$. Consequently, $ A\tilde{x} - b - z\in Q$.
Moreover, by setting $ y = 0_{X^*}$ and $w = 0_{Z^*}$ in (\ref{theo:main-algebra-equa4}) we arrive at 
$r\ge v_{\mathcal{D}}(z) \geq \langle \tilde{x},  c  \rangle$. 
From \eqref{Definition_H}, it is now clear that $(z,r)\in \mathcal{H}$. The proof is complete.
\end{proof}

Similarly to Proposition \ref{prop:properties-set-T}, one easily proves the following assertions.

\begin{proposition}\label{prop:properties-set-U}
Let $y \in \dom v_{\mathcal{P}}$. Define 
\begin{equation}\label{U}
\mathcal{H}_y:=\displaystyle\bigcup_{x\,\in\, P}\big((Ax-b-Q)\times[\langle x,c-y\rangle,+\infty[\big).
\end{equation}
Then we have the following assertions:
\begin{enumerate}
\item[{\rm (i)}] $\mathcal{H}_y$ is a nonempty convex subset of $Z\times \R$ and there exists $\ox\in P$ such that
\begin{equation}\label{QtimesinU}
   -Q \times [\la \ox,c-y \ra , + \infty[\, \subset\, \mathcal{H}_y. 
\end{equation}
Consequently,
\begin{equation}\label{aa}
   (-\cor Q)\, \times\, ]\langle\bar{x}, c-y\rangle,+\infty[ \,\,\subset\, \cor \mathcal{H}_y,
\end{equation}
\item[{\rm (ii)}] $(0_{Z}, v_{\mathcal{P}}(y)) \notin \cor \mathcal{H}_y$,
\item[{\rm (iii)}] For any $\epsilon>0$, we have  $(0_{Z}, v_{\mathcal{P}}(y) + \epsilon) \in  \mathcal{H}_y$,
\item[{\rm (iv)}] Suppose that 
\begin{equation}\label{condition-core-P}
    \exists\, \ox \in P \text{ such that } A\ox - b \in \cor Q.
\end{equation}
Then for any $z \in Z$ there exists $\delta_0 >0$  such that  $\left(\delta z, \langle \ox,c-y \rangle\right) \in  \mathcal{H}_y$ whenever $|\delta|\leq\delta_0$.
\end{enumerate}    
\end{proposition}

\begin{theorem}\label{theo:main-algebra*}
If \eqref{condition-core-P} holds, then $\mathcal{K}=\mathcal{M}$.
\end{theorem}
\begin{proof}
The proof of this theorem bases on the arguments similar to that of Theorem \ref{theo:main-algebra}. Let us take arbitrarily $(y, r) \in \mathcal{M}$ and show that $(y,r)\in \mathcal{K}$. By the definition of $\mathcal{M}$, one has $y \in \operatorname{dom} v_{\mathcal{P}}$ and $r \leq v_{\mathcal{P}}(y)$.

Observe by Proposition \ref{prop:properties-set-U}(i) that $\mathcal{H}_y$ from \eqref{U} is a nonempty convex subset in $Z \times \mathbb{R}$. Further, $\cor \mathcal{H}_y \neq \emptyset$ and $\left(0_Z, v_{\mathcal{P}}(y)\right) \notin \operatorname{core} \mathcal{H}_y$. Theorem \ref{theo:sepacore} now shows that there exists $\left(\bar{z}^*, \bar{\alpha}\right) \in Z^* \times \mathbb{R}$ such that $\left(\bar{z}^*, \bar{\alpha}\right) \neq\left(0_{Z^*}, 0\right)$ and
\begin{equation}\label{separate-from-U}
\bar{\alpha} v_{\mathcal{P}}(y) \geq\left\langle z, \bar{z}^*\right\rangle+\bar{\alpha}\alpha \quad \text{ for all }(z, \alpha) \in \mathcal{H}_y.
\end{equation}
Proposition \ref{prop:properties-set-U}(iii) ensures that $\left(0_Z, v_{\mathcal{P}}(y)+\epsilon\right) \in \mathcal{H}_y$ for any $\e>0$. This together with \eqref{separate-from-U} gives
$$
\bar{\alpha} v_{\mathcal{P}}(y) \geq \bar{\alpha}\left(v_{\mathcal{P}}(y)+\epsilon\right),
$$
which implies that $\bar{\alpha} \leq 0$. We proceed to show that $\bar{\alpha}<0$. To obtain a contradiction, suppose that $\bar{\alpha}=0$. Then \eqref{separate-from-U} becomes
$$
0 \geq\left\langle z, \bar{z}^*\right\rangle \quad \forall(z, \alpha) \in \mathcal{H}_y.
$$
By using the same method as in the proof of Theorem \ref{theo:main-algebra} we get that $\bar{z}^*=0_{Z^*}$, contrary to $\left(\bar{z}^*, \bar{\alpha}\right) \neq\left(0_{Z^*}, 0\right)$. Therefore, $\bar{\alpha}<0$. Without loss of generality we can assume $\bar{\alpha}=-1$, and so \eqref{separate-from-U} can be rewritten as
\begin{equation}\label{normalize-a=1}
v_{\mathcal{P}}(y) \leq -\left\langle z, \bar{z}^*\right\rangle+\alpha \quad \text{ for all }(z, \alpha) \in \mathcal{H}_y .
\end{equation}
We next prove that $\bar{z}^* \in Q^*$. To do this, fix $z \in Q$. Taking arbitrarily $\eta>0$. Since $Q$ is a cone, it follows from \eqref{QtimesinU} that $(-\eta z,\langle\bar{x}, c-y\rangle) \in \mathcal{H}_y$. From this and \eqref{normalize-a=1} we have
$$
v_{\mathcal{P}}(y) \leq\left\langle\eta z, \bar{z}^*\right\rangle+\langle\bar{x}, c-y\rangle\quad \text { for all } \eta>0,
$$
which yields
$$
\frac{1}{\eta} v_{\mathcal{P}}(y) \leq\left\langle z, \bar{z}^*\right\rangle+\frac{1}{\eta}\langle\bar{x}, c-y\rangle .
$$
Letting $\eta \longrightarrow \infty$ gives $\left\langle z, \bar{z}^*\right\rangle \geq 0$. As this inequality holds for any $z \in Q$, we have $\bar{z}^* \in Q^*$. 

Our next claim is that $-A^* \bar{z}^*+c-y \in P^*$. Let us fix an arbitrary $x \in P$. It is clear that $\big(A(\lambda x)-b,\langle\lambda x, c-y\rangle\big) \in \mathcal{H}_y$ for any $\lambda>0$ because $P$ is a cone. Combining this with \eqref{normalize-a=1} we obtain
\begin{equation}\label{eq:516}
v_{\mathcal{P}}(y) \leq\left\langle-A(\lambda x)+b, \bar{z}^*\right\rangle+\langle\lambda x, c-y\rangle \text { for all } \lambda>0,
\end{equation}
which implies that
$$
\frac{1}{\lambda} v_{\mathcal{P}}(y) \leq\left\langle x,-A^* \bar{z}^*+c-y\right\rangle+\frac{1}{\lambda}\left\langle b, \bar{z}^*\right\rangle .
$$
By letting $\lambda \longrightarrow \infty$, we have $\left\langle x,-A^* \bar{z}^*+c-y\right\rangle \geq 0$. It is worth noting that this inequality holds for any $x \in P$. Hence, $-A^* \bar{z}^*+c-y \in P^*$, or $y \in-A^* \bar{z}^*+c-P^*$. On the other hand, substituting $x=0_X$ into \eqref{eq:516} we get $v_{\mathcal{P}}(y) \leq\left\langle b, \bar{z}^*\right\rangle$. This completes the proof.
\end{proof}

\section{Topological characterizations of (D) and (D$^*$)}\label{sec:topo}
So far, we've limited our exploration to a purely algebraic context, aiming to minimize the incorporation of additional structure. One consequence we face is the necessity to formulate our dual problem within the algebraic dual of the primal space, which seems to be a difficult space to work with. To facilitate this process and, in addition, to employ powerful tools from \textit{functional analysis} such as separation theorems and continuous extension theorems, we are inclined to endow the spaces in consideration with topological structures.

The following concepts and related facts can be found from any book focusing on topological vector spaces, e.g., Robertson and Robertson \cite{Robertson64} and Schaefer \cite{Schaefer71}. Given a \textit{separated dual pair} of vector spaces $(X,Y)$, the notion of which is defined in Section \ref{sec:prelim}. The collection of sets
\begin{equation*}
    \left\{x\in X: \left| \la x,y\ra \right|\le 1 \quad \forall y\in S\right\}, \text{ where }S \text{ is a } \text{finite subset of }Y,
\end{equation*}
forms a base of neighborhoods of the origin in $X$. The resulting topology is then called the \textit{weak topology} on $X$, denoted by $\sigma (X,Y)$. A topological vector space is \textit{locally convex} if it has a base of convex sets. A locally convex topology on $X$ is \textit{consistent} with the dual pair $(X,Y)$ if $Y$ is exactly the dual of $X$ with respect to this topology. Among such topologies on $X$, $\sigma (X,Y)$ is the coarsest one while the \textit{Mackey topology}, denoted by $\tau (X,Y)$, is the finest one. One may find the construction of the Mackey topology from \cite[p.~131]{Schaefer71}. When $X$ is a normed space and $Y$ coincides with its topological dual, the weak topology $\sigma (X,Y)$ and Mackey topology $\tau (X,Y)$ are exactly the usual weak topology $\sigma (X,X^*)$ and norm/strong topology on $X$, respectively.

Let $(Z, W)$ be another separated dual pair of vector spaces. Let $\sigma(X,Y)$ and $\sigma(Z,W)$ be the weak topologies on $X$ and $Z$, respectively, and let $A$ be a $\sigma(X, Y)-\sigma(Z, W)$ continuous linear map from $X$ to $Z$. Using \cite[Proposition~4,~p.~37]{AndersonNash87}, we know that the \textit{adjoint (transpose)} of $A$ is the linear map
$A^*:W\rightarrow Y$ defined via the condition
\begin{equation}\label{adjoint}
\langle x, A^*w\rangle=\langle Ax,w\rangle\quad \forall x\in X,\;\forall w\in W.
\end{equation}
It is well known \cite[Chapter~IV,~Section~2]{Schaefer71} that $A^*$ is $\sigma(W, Z)-\sigma(Y, X)$ continuous, where $\sigma(W, Z)$ and $\sigma(Y, X)$ are weak topologies on $W$ and $Y$, respectively. We adopt the convention that every topological notion being used in the sequel is understood with respect to the appropriate weak topology.

As $\mathcal{H}$ (resp. $\mathcal{K}$) is convex, recall by \cite[3.1,~p.~130]{Schaefer71} that the closure of $\mathcal{H}$ (resp. $\mathcal{K}$) is the same for any locally convex topology on $Z\times \R$ (resp. $Y\times \R$) consistent with the pair $(Z\times \R,W\times \R)$ (resp. $(Y\times \R,X\times \R)$). In particular, the closures of $\mathcal{H}$ and $\mathcal{K}$ with respect to the weak topology and Mackey topology are the same. Surprisingly, when $Z$ and $Y$ are endowed with their weak topologies, the weak closure $\overline{\mathcal{H}}$ becomes exactly $\mathcal{N}$, and $\overline{\mathcal{K}}$ becomes $\mathcal{M}$. 

Observe also that in our current topological setting, the function $v_{\mathcal{P}}$ is \textit{upper semicontinuous}, i.e., each upper level set $\{y\in Y: v_{\mathcal{P}}(y)\ge \alpha\}$ is closed for each $\alpha \in \R$; and the function $v_{\mathcal{D}}$ is \textit{lower semicontinuous}, i.e., each lower level set $\{z\in Z: v_{\mathcal{D}}(z)\le \alpha\}$ is closed for each $\alpha \in \R$. These observations later shed light on the closedness of $\mathcal{N}$ and $\mathcal{M}$.  As promised, we now show that $\mathcal{N}$ (resp. $\mathcal{M}$) is exactly the closure of $\mathcal{H}$ (resp. $\mathcal{K}$), which thus implies that condition {\bf (D)} (resp. {\bf (D$^*$)}) is fulfilled if and only if $\mathcal{H}$ (resp. $\mathcal{K}$) is closed with respect to consistent topologies.
\begin{theorem}
\label{lemmathesetHN}
	The following assertions hold true:
	\begin{description}
		\item[{\rm (i)}] $\overline{\mathcal{H}} = \mathcal{N}$,		
		\item[{\rm (ii)}] 	If $P,Q$ are closed, then $\overline{\mathcal{K}} = \mathcal{M}$.
	\end{description}
\end{theorem}
\begin{proof}
(i) First note that $\mathcal{N}=\epi\, (v_{\mathcal{D}})$ is closed with respect to the $\sigma (Z\times \R,W\times \R)$ topology since $v_{\mathcal{D}}$ is lower semicontinuous. Then, the inclusion $\overline{\mathcal{H}} \subset \mathcal{N}$ follows from Proposition \ref{prop:MN}(i) and the closedness of $\mathcal{N}$.

For the converse inclusion, let $(z,r)\in \mathcal{N}$ such that $(z,r)\notin \overline{\mathcal{H}}$. Then by the separation theorem \cite[Proposition~6,~p.~37]{AndersonNash87} there exists $(w,\beta)\in W \times \R$ such that 
\begin{equation*}
 \la (z,r),(w,\beta)\ra > \la (z',r'),(w,\beta)\ra \quad \forall (z',r')\in \mathcal{H},
\end{equation*}
which, by Theorem \ref{contra}, contradicts $(z,r)\in \mathcal{N}$. Thus we conclude that $(z,r)\in \overline{\mathcal{H}}$. Therefore, $\overline{\mathcal{H}} = \mathcal{N}$. 

The proof of (ii) follows the same pattern as that of (i) by interchanging the roles of primal and dual settings.
\end{proof}	

The following example serves as an illustration for the identity $\overline{\mathcal{H}}=\mathcal{N}$ which is shown previously.

\begin{example}\label{ex:6.1}
\rm 
We consider again the parametric version of Gale's example presented in Example \ref{exam:Gale1}. Therein, it is explicitly shown that
\begin{equation*}
\begin{aligned}
    \mathcal{N}=\bigcup_{z\,\in\, \mathrm{dom}\, v_{\mathcal{D}}}(\{z\}\times [v_{\mathcal{D}}(z),+\infty[)= \big\{(t_1+t_2-\alpha,t_2-\beta,t_3):t_1,t_2,t_3\ge 0\big\}.
\end{aligned}
\end{equation*}
On the other hand, the closure $\overline{\mathcal{H}}$ of $\mathcal{H}=\bigcup_{x\,\in\, P}\big((Ax-b-\{0_{\mathbb{R}^2}\})\times [\la x,c\ra ,+\infty[\big)$ is calculated in \cite[Example~3.1]{KhanhMoTrinh} as 
\begin{equation}\label{Hbar}
    \overline{\mathcal{H}}=\big\{(t_1+t_2-\alpha,t_2-\beta,t_3):t_1,t_2,t_3\ge 0\big\}
\end{equation}
which is exactly $\mathcal{N}$. This is in accordance with the result of Theorem \ref{lemmathesetHN}.
\end{example}

\section{Concluding Remarks}\label{sec:con}
In this paper, we have established algebraic versions of Farkas lemma and strong duality results in the general framework of infinite-dimensional conic linear programming. By analyzing the algebraic properties of perturbed optimal value functions and their associated hypergraphical and epigraphical sets, we have developed novel perturbed Farkas-type lemmas, which, in turn, provided complete characterizations of the ``zero duality gap'' in a very broad setting. By means of Section \ref{sec:algebra} and Section \ref{sec:topo}, our work has shown to provide a cohesive framework for understanding the theory of Farkas lemma and strong duality results.

\section*{Acknowledgments}
This research is funded by Ho Chi Minh City Open University under the grant number E2024.08.2.

\end{document}